\documentclass{elsart}
\usepackage{amsfonts}
\usepackage{makeidx}
\usepackage{graphicx}
\usepackage[english,activeacute]{babel}
\usepackage[latin1]{inputenc}
\usepackage{amsmath,amssymb,latexsym,comment}
\usepackage[numbers,sort&compress]{natbib}

\bibpunct{[}{]}{,}{n}{}{;}
\newtheorem{theorem}{Theorem}[section]

\newtheorem{lemma}{Lemma}[section]

\newtheorem{proposition}{Proposition}[section]
\newtheorem{remark}{Remark}[section]

\newenvironment{proof}[1][Proof]{\noindent\textbf{#1.} }{\ \rule{0.5em}{0.5em}}
\begin{document}

\begin{frontmatter}

\title{Asymptotics for Laguerre--Sobolev type orthogonal polynomials modified within their oscillatory regime}

\author[C] {Edmundo J. Huertas \thanksref{E}}
\author[M] {Francisco Marcell\'an \thanksref{P}}
\author[M] {M. Francisca P\'erez-Valero  \thanksref{F}}
\author[V] {Yamilet Quintana \thanksref{Y}}

\address[C] {CMUC, Departamento de Matem\'atica, FCTUC, Universidade de Coimbra, Largo D. Dinis, Apdo. 3008, 3001-544 Coimbra, Portugal}

\address[M] {Departamento de Matem\'aticas, Universidad Carlos III, Legan\'es-Madrid, Spain}

\address[V] {Departamento de Matem\'aticas Puras y Aplicadas, Universidad Sim\'on Bolívar, Postal Code 89000, Caracas 1080 A, Venezuela}

\thanks[E] {Corresponding author. Partially supported by Dirección General de Investigación Científica y Técnica, Ministerio de Economía y Competitividad of Spain,
grant MTM 2012-36732-C03-01, and Fundac\~ao para a Ci\^encia e Tecnologia (FCT) of Portugal, ref. SFRH/BPD/91841/2012.}

\thanks[P] {Supported by Direcci\'on General de Investigaci\'on Cient\'ifica, Ministerio de Econom\'ia y Competitividad of Spain, grant MTM 2012-36732-C03-01.}

\thanks[F] {Supported by the Research Fellowship Program, Ministerio de Ciencia e Innovaci\'on (MTM 2009-12740-C03-01) and Direcci\'on General de Investigaci\'on Cient\'ifica, Ministerio de Econom\'ia y Competitividad of Spain, grant MTM 2012-36732-C03-01.}

\thanks[Y] {Partially supported by Direcci\'on General de Investigaci\'on Cient\'ifica, Ministerio de Econom\'ia y Competitividad of Spain, grant MTM 2012-36732-C03-01.}

\begin{abstract}
In this paper we consider sequences of polynomials orthogonal with respect
to the discrete Sobolev inner product
\begin{equation*}
\left\langle f,g\right\rangle_{S} =\int_{0}^{\infty }f(x)g(x)x^{\alpha
}e^{-x}dx+\mathbb{F}(c)A\mathbb{G}(c)^{t}, \, \, \, \alpha>-1,
\end{equation*}
where $f$ and $g$ are polynomials with real coefficients, $A \in \mathbb{R }%
^{(2,2)}$ and the vectors $\mathbb{F}(c),\,\mathbb{G}(c)$ are
\begin{equation*}
A=\left(%
\begin{array}{cc}
M & 0 \\
0 & N%
\end{array}%
\right),\quad \mathbb{F}(c)= (f(c), f^{\prime }(c)) \mbox{ and } \mathbb{G}%
(c)= (g(c), g^{\prime }(c)), \mbox{ respectively,}
\end{equation*}
with $M,N\in \mathbb{R}_{+}$ and the mass point $c$ is located inside the
oscillatory region for the classical Laguerre polynomials. We focus our
attention on the representation of these polynomials in terms of
classical Laguerre polynomials and we deduce analyze the behavior of the coefficients of the
corresponding five-term recurrence relation when the degree of the polynomials is large enough. Also, the outer relative asymptotics of the Laguerre-Sobolev type with respect to the Laguerre polynomials is analyzed.

\end{abstract}


\end{frontmatter}

%%%%%%%%%%%%%%%%%%%%%%%%%%%%%%%%%%%%%%%%%%%%%%%%%%%%%%%%%%%%%%%%%%%%%%%%%%%%%%%%%%

\section{Introduction}

\label{[Section-1]-Intro}

%%%%%%%%%%%%%%%%%%%%%%%%%%%%%%%%%%%%%%%%%%%%%%%%%%%%%%%%%%%%%%%%%%%%%%%%%%%%%%%%%%

The study of asymptotic properties for general orthogonal polynomials is an
important challenge in approximation theory and their applications permeate
many fields in science and engineering \citep{N79,O99,ST92,Szego75}.
Although it may seem as an old subject from the point of view of standard
orthogonality \citep{Chi78,Szego75}, this is not the case neither in the
general setting (cf. \citep{L88,LMVA95,N79,R77,R82,R87,ST92}) nor from the
viewpoint of Sobolev orthogonality, where it remains like a partially
explored subject \citep{BMJL06}. In fact, in the last ten years this topic
has attracted the interest of many researchers
\citep{BFM02,DM10,DM11,DHM11,DHM12,FZ2010,FMP99,MM06,MPQU13,MQU13,MQU12,
MZFH12,PRT10,PQRT10,PQRT11}.

A Sobolev-type or discrete Sobolev-type inner product on the linear space $%
\mathbb{P}$ of polynomials with real coefficients is defined by
\begin{equation}
\left\langle f,g\right\rangle _{S}=\int f(x)g(x)d\mu _{0}(x)+\sum_{k=0}^{d}%
\mathbb{F}(c_{k})A_{k}\mathbb{G}(c_{k})^{t},\,d\in \mathbb{Z}_{+},
\label{[Sec1]-SobtyInnPr}
\end{equation}%
where $\mu _{0}$ is a nontrivial finite and positive Borel measure supported
on the real line, $f,g\in \mathbb{P}$, and for $k=0,\ldots ,d$, $d\in
\mathbb{Z}_{+}$, the matrices $A_{k}=(a_{ij}^{(k)})\in \mathbb{R}%
^{(1+N_{k})(1+N_{k})}$ are positive semi-definite. We denote by $\mathbb{F}%
(c_{k})$ and $\mathbb{G}(c_{k})$ the vectors $\mathbb{F}(c_{k})=\left(
f(c_{k}),f^{\prime }(c_{k}),\ldots ,f^{(N_{k})}(c_{k})\right) $ and $\mathbb{%
G}(c_{k})=\left( g(c_{k}),g^{\prime }(c_{k}),\ldots
,g^{(N_{k})}(c_{k})\right) $, respectively, with $c_{k}\in \mathbb{R}$, $%
N_{k}\in \mathbb{Z}_{+}$ and, as usual, $v^{t}$ denotes the transpose of the vector $v$%
. This notion was initially introduced in \citep{ELMMR95} for diagonal
matrices $A_{k}$ in order to study recurrence relations for sequences of
polynomials orthogonal with respect to (\ref{[Sec1]-SobtyInnPr}).

The study of asymptotic properties of the sequences of orthogonal
polynomials with respect to particular cases of the inner product (\ref%
{[Sec1]-SobtyInnPr}) has been done by considering separately the cases `mass points
inside' or `mass points outside' of $\mbox{{\em supp\/}}\mu _{0}$,
respectively, being $\mbox{{\em supp\/}}\mu _{0}$ a bounded interval of $%
\mathbb{R}$ or, more recently, an unbounded interval of the real line (see,
for instance \citep{DM10,DM11,DHM11,DHM12,FZ2010,MM06,MZFH12}). The first
results in the literature about asymptotic properties of orthogonal
polynomials with respect to a Sobolev-type inner product like (\ref%
{[Sec1]-SobtyInnPr}) appear in \citep{MVA93}, where the authors considered $%
d=0$, $N_{0}=1$, $a_{11}^{(0)}=a_{12}^{(0)}=a_{21}^{(0)}=0,$ $%
a_{22}^{(0)}=\lambda $, with $\lambda >0$. Therein, such asymptotic
properties when there is only one mass point supporting the derivatives
either inside or outside $[-1,1]$ and $\mu $ is a measure in the Nevai class
$M(0,1)$ are studied.

In \citep{LMVA95}, using an approach based on the theory of Padé
approximants, the authors obtain the outer relative asymptotics for
orthogonal polynomials with respect to the Sobolev-type inner product (\ref%
{[Sec1]-SobtyInnPr}) assuming that $\mu _{0}$ belongs to Nevai class $M(0,1)$
and the mass points $c_{k}$ belong to $\mathbb{C}\setminus
\mbox{{\em
supp\/}}\mu $. The same problem with the mass points in $\mbox{{\em supp\/}}%
\mu =[-1,1]$ was solved in \citep{MRS03}, provided that $\mu ^{\prime }(x)>0$
a.e. $x\in \lbrack -1,1]$ and $A_{k}$ being diagonal matrices with $%
a_{ii}^{(k)}$ non-negative constants. The pointwise convergence of the
Fourier series associated to such an inner product was studied when $\mu
_{0} $ is the Jacobi measure (see also \citep{MOR02,MOR022}). On the other hand, the
asymptotics for orthogonal polynomials with respect to the Sobolev-type
inner product (\ref{[Sec1]-SobtyInnPr}) with $\mu _{0}\in M(0,1)$, $c_{k}$
belong to $\mbox{{\em supp\/}}\mu \setminus \lbrack -1,1],$ and $A_{k}$ are
complex diagonal matrices such that $a_{1+N_{k},1+N_{k}}^{(k)}\not=0$, was
solved in \citep{AS05}.

Another results about the asymptotic behavior of orthogonal
polynomials associated with diagonal (resp. non-diagonal) Sobolev inner
products with respect to measures supported on the complex plane can be
found in \citep{AMRR95,BFM02,DM10,LP99}. On the other hand, results concerning
asymptotics for extremal polynomials associated to non-diagonal Sobolev
norms may be seen in \citep{LPP05,PRT10,PQRT10,PQRT11}.

In this paper we deal with sequences of polynomials orthogonal with respect
to a particular case of (\ref{[Sec1]-SobtyInnPr}). Indeed, $\mu _{0}$ is the
Laguerre classical measure%
\begin{equation}
\langle f,g\rangle _{S}=\int_{0}^{\infty }f(x)g(x)x^{\alpha }e^{-x}dx+%
\mathbb{F}(c)A\mathbb{G}(c)^{t}, \alpha >-1,  \label{[Sec1]-DicrLagSob}
\end{equation}%
$f,g\in \mathbb{P}$. The matrix $A$ and the vectors $\mathbb{F}(c),\,\mathbb{%
G}(c)$ are
\begin{equation*}
A=\left(
\begin{array}{cc}
M & 0 \\
0 & N%
\end{array}%
\right) ,\quad \mathbb{F}(c)=(f(c),f^{\prime }(c))\mbox{ and }\mathbb{G}%
(c)=(g(c),g^{\prime }(c)),\mbox{ respectively,}
\end{equation*}%
$M,N\in \mathbb{R}_{+}$, and the mass point $c$ is located inside the
oscillatory region for the classical Laguerre polynomials, i.e., $c>0$.
Following the methodology given in \citep{DM10,DM11,DHM11,DHM12,MM06,MZFH12}
we focus our attention on the representation of these polynomials in terms
of the classical Laguerre polynomials. Their asymptotic behavior will be
discussed.

More precisely, as it was mentioned above, recent works like %
\citep{DM10,DM11,DHM11,DHM12,MM06,MZFH12} have focused the attention on the
study of asymptotic properties of the sequences of orthogonal polynomials
with respect to specific cases of the inner product (\ref{[Sec1]-SobtyInnPr}%
) with `mass points outside' of $\mbox{{\em supp\/}}\mu _{0}$, being $%
\mbox{{\em supp\/}}\mu _{0}$ an unbounded interval of the real line.
However, to the best of our knowledge, asymptotic properties of the
sequences of orthogonal polynomials associated to (\ref{[Sec1]-DicrLagSob})
are not available in the literature.

The structure of the manuscript is as follows. Section \ref%
{[Section-2]-Background} contains the basic background about Laguerre
polynomials and some other auxiliary results which will be used throughout
the paper. In Section \ref{[Section-3]-OuterAs} we prove our main result,
namely the outer relative asymptotic of the Laguerre-Sobolev type orthogonal
polynomials modified into the \textit{positive} real semiaxis. Finally, in
Section \ref{[Section-4]-5TRR} we deduce the coefficients of the
corresponding five-term recurrence relation as well as their asymptotic
behavior when the degree of the polynomials is large enough.

Throughout this manuscript, the notation $u_{n}\sim v_{n}$ means that the
sequence $\{\frac{u_{n}}{v_{n}}\}_{n}$ converges to certain non zero
constant as $n\rightarrow \infty$. Any other standard
notation will be properly introduced whenever needed.

%%%%%%%%%%%%%%%%%%%%%%%%%%%%%%%%%%%%%%%%%%%%%%%%%%%%%%%%%%%%%%%%%%%%%%%%%%%%%%%%

\section{Background and previous results}

\label{[Section-2]-Background}

%%%%%%%%%%%%%%%%%%%%%%%%%%%%%%%%%%%%%%%%%%%%%%%%%%%%%%%%%%%%%%%%%%%%%%%%%%%%%%%%

Laguerre orthogonal polynomials are defined as the polynomials orthogonal
with respect to the inner product
\begin{equation}
\left\langle f,g\right\rangle _{\alpha }=\int_{0}^{\infty
}f(x)g(x)\,x^{\alpha }e^{-x}dx,\quad \alpha >-1,\quad f,g\in \mathbb{P}.
\label{[Sec2]-LagInnPr}
\end{equation}

The expression of these polynomials as an $\,_{1}F_{1}$ hypergeometric
function is very well known in the literature (see for instance, %
\citep{IsmBk05,NikUv88,Szego75}). The connection between these two facts
follows from a characterization of such orthogonal polynomials as
eigenfunctions of a second order linear differential operator with
polynomial coefficients. The following proposition will be useful in the
sequel and it summarizes some structural and asymptotic properties of
Laguerre polynomials involving two different normalizations %
\citep{Chi78,MBP94,Szego75}.

\begin{proposition}
\label{[Sec2]-PROP-21} Let $\{\widehat{L}_{n}^{\alpha }(x)\}_{n\geq 0}$ be
the sequence of monic Laguerre orthogonal polynomials. Then the following
statements hold.

\begin{enumerate}
\item Three-term recurrence relation. For every $n\geq 1$%
\begin{equation}
x\widehat{L}_{n}^{\alpha }(x)=\widehat{L}_{n+1}^{\alpha }(x)+\beta _{n}%
\widehat{L}_{n}^{\alpha }(x)+\gamma _{n}\widehat{L}_{n-1}^{\alpha }(x),
\label{[Sec2]-3TRRLag}
\end{equation}%
with initial conditions $\widehat{L}_{0}^{\alpha }(x)=1$, $\widehat{L}%
_{1}^{\alpha }(x)=x-(\alpha +1)$, and $\beta
_{n}=2n+\alpha +1$, $\gamma _{n}=n(n+\alpha )$.

\item For every $n\in \mathbb{N}$,%
\begin{equation}
||\widehat{L}_{n}^{\alpha }||_{\alpha }^{2}=\Gamma (n+1)\Gamma (n+\alpha +1).
\label{[Sec2]-LagNorm}
\end{equation}

\item Hahn's condition. For every $n\in \mathbb{N}$,%
\begin{equation}
\lbrack \widehat{L}_{n}^{\alpha }(x)]^{\prime }=n\widehat{L}_{n-1}^{\alpha
+1}(x).  \label{[Sec2]-LagDer}
\end{equation}

\item The $n$-th Dirichlet kernel $K_{n}(x,y)$, given by%
\begin{equation}
K_{n}(x,y)=\sum_{k=0}^{n}\frac{\widehat{L}_{k}^{\alpha }(x)\widehat{L}%
_{k}^{\alpha }(y)}{||\widehat{L}_{k}^{\alpha }||_{\alpha }^{2}},
\label{[Sec2]-KLagxy}
\end{equation}%
satisfies the Christoffel-Darboux formula (cf. \citep[Theorem 3.2.2]{Szego75}%
):
\begin{equation}
K_{n}(x,y)=\frac{1}{||\widehat{L}_{n}^{\alpha }||_{\alpha }^{2}}\left( \frac{%
\widehat{L}_{n+1}^{\alpha }(x)\widehat{L}_{n}^{\alpha }(y)-\widehat{L}%
_{n}^{\alpha }(x)\widehat{L}_{n+1}^{\alpha }(y)}{(x-y)}\right) ,\quad n\geq
0.  \label{[Sec2]-CrDarbK}
\end{equation}

\item The so called confluent form of the above kernel is given by%
\begin{equation}
K_{n}(x,x)=\frac{1}{||\widehat{L}_{n}^{\alpha }||_{\alpha }^{2}}\left\{ [%
\widehat{L}_{n+1}^{\alpha }]^{\prime }(x)\widehat{L}_{n}^{\alpha }(x)-[%
\widehat{L}_{n}^{\alpha }]^{\prime }(x)\widehat{L}_{n+1}^{\alpha
}(x)\right\} ,\quad n\geq 0.  \label{[Sec2]-KLagxx}
\end{equation}

\item Let $\{L_{n}^{(\alpha )}(x)\}_{n\geq 0}$ be the sequence of Laguerre
orthogonal polynomials with leading coefficient $\frac{\left( -1\right) ^{n}%
}{n!}$, then
\begin{equation}
L_{n}^{(\alpha )}\left( x\right) =\frac{\left( -1\right)^{n}}{n!}\widehat{L}%
_{n}^{\alpha }(x).  \label{[Sec2]-(-1)nL}
\end{equation}

\item \citep[Theorem 8.22.3]{Szego75} Outer strong asymptotics or Perron
asymptotics formula on $\mathbb{C}\setminus \mathbb{R}_{+}$. Let $\alpha \in
\mathbb{R}$, then%
\begin{eqnarray}
L_{n}^{(\alpha )}\left( x\right) &=&\frac{1}{2}\pi ^{-1/2}e^{x/2}\left(
-x\right) ^{-\alpha /2-1/4}n^{\alpha /2-1/4}\exp \left( 2\left( -nx\right)
^{1/2}\right)  \label{[Sec2]-Perron} \\
&&\quad \hspace{20pt}\hspace{50pt}\times \left\{
\sum\limits_{k=0}^{p-1}C_{k}(\alpha ;x)n^{-k/2}+\mathcal{O}%
(n^{-p/2})\right\} .  \notag
\end{eqnarray}

Here $C_{k}(\alpha ;x)$ is independent of $n$. This relation holds for $x$
in the complex plane with a cut along the positive real semiaxis, and it
also holds if $x$\ is in the cut plane mentioned. $\left( -x\right)
^{-\alpha /2-1/4}$ and $\left( -x\right) ^{1/2}$ must be taken real and
positive if $x<0$. The bound for the remainder holds uniformly in every
compact subset of the complex plane with empty intersection with $\mathbb{R}%
_{+}$.

\item \citep[Theorem 8.22.2]{Szego75} Perron generalization of Fejér formula
on $\mathbb{R}_{+}$. Let $\alpha \in \mathbb{R}$. Then for $x>0$ we have
\begin{equation}
\begin{array}{rr}
L_{n}^{(\alpha )}\left( x\right) & =\pi ^{-1/2}e^{x/2}x^{-\alpha
/2-1/4}n^{\alpha /2-1/4}\cos \{2\left( nx\right) ^{1/2}-\alpha \pi /2-\pi
/4\} \\
& \cdot \left\{ \sum\limits_{k=0}^{p-1}A_{k}(x)n^{-k/2}+\mathcal{O}%
(n^{-p/2})\right\} \\
& +\pi ^{-1/2}e^{x/2}x^{-\alpha /2-1/4}n^{\alpha /2-1/4}\sin \{2\left(
nx\right) ^{1/2}-\alpha \pi /2-\pi /4\} \\
& \cdot \left\{ \sum\limits_{k=0}^{p-1}B_{k}(x)n^{-k/2}+\mathcal{O}%
(n^{-p/2})\right\} ,%
\end{array}
\label{[Sec2]-Fejer}
\end{equation}%
where $A_{k}(x)$ and $B_{k}(x)$ are certain functions of $x$\ independent of
$n$ and regular for $x>0$. The bound for the remainder holds uniformly in $%
[\epsilon ,\omega ]$. For $k=0$ we have $A_{0}(x)=1$ and $B_{0}(x)=0$.
\end{enumerate}
\end{proposition}

Next, we summarize some results about the so called $k$--iterated
Christoffel perturbed Laguerre orthogonal polynomials. They are orthogonal
with respect to the inner product%
\begin{equation}
\langle f,g\rangle _{\lbrack k]}=\int_{0}^{\infty
}f(x)g(x)(x-c)^{k}x^{\alpha }e^{-x}dx,\quad \alpha >-1,\,\,f,g\in \mathbb{P},
\label{[Sec2]-kChrisLagInn}
\end{equation}%
and we will denote by $\{\widehat{L}_{n}^{\alpha ,[k]}(x)\}_{n\geq 0}$ the
corresponding monic sequence and by%
\begin{equation*}
||\widehat{L}_{n}^{\alpha ,[k]}||_{[k]}^{2}=\langle \widehat{L}_{n}^{\alpha
,[k]}(x),x^{n}\rangle _{\lbrack k]}
\end{equation*}%
the norm of the $n$--th degree polynomial in the sequence. Note that the
modified Laguerre measure $(x-c)^{k}x^{\alpha }e^{-x}dx$ is positive
definite when either $k$ is an even integer or $k$ is an odd number and $c$ is
outside the interval $[0,+\infty )$. It is very well known that, when $%
k=1$ and $c$ is \textit{outside} the support of the classical Laguerre
measure, i.e., when it is assumed that $c$ is not a zero of $\widehat{L}%
_{n}^{\alpha }(x)$, these polynomials are actually the monic Laguerre
Kernels (\ref{[Sec2]-CrDarbK}) (see \citep[Sec. I.7]{Chi78}).

We introduce the following standard notation for the partial derivatives of the
$n$-th Dirichlet kernel $K_{n}(x,y)$%
\begin{equation*}
\frac{\partial ^{j+k}K_{n}(x,y)}{\partial x^{j}\partial y^{k}}%
=K_{n}^{(j,k)}(x,y),\quad 0\leq i,j\leq n.
\end{equation*}%
Taking derivatives with respect to $y$ in (\ref{[Sec2]-KLagxy}) and considering $x=y=c$ we get%
\begin{equation}
K_{n-1}^{(0,1)}(c,c)=\frac{1}{2}\frac{\widehat{L}_{n-1}^{\alpha }(c)[%
\widehat{L}_{n}^{\alpha }]^{\prime \prime }(c)-\widehat{L}_{n}^{\alpha }(c)[%
\widehat{L}_{n-1}^{\alpha }]^{\prime \prime }(c)}{\Gamma (n)\Gamma (n+\alpha
)}.  \label{[Sec2]-K01cc}
\end{equation}

On the other hand,%
\begin{equation*}
K_{n-1}^{(1,1)}(c,c)=\frac{1}{3!}\frac{1}{\Gamma (n)\Gamma (n+\alpha )}\cdot
\end{equation*}%
\begin{equation}
\left\{ \widehat{L}_{n-1}^{\alpha }(c)[\widehat{L}_{n}^{\alpha }]^{\prime
\prime \prime }(c)+3[\widehat{L}_{n-1}^{\alpha }]^{\prime }(c)[\widehat{L}%
_{n}^{\alpha }]^{\prime \prime }(c)-\widehat{L}_{n}^{\alpha }(c)[\widehat{L}%
_{n-1}^{\alpha }]^{\prime \prime \prime }(c)-3[\widehat{L}_{n}^{\alpha
}]^{\prime }(c)[\widehat{L}_{n-1}^{\alpha }]^{\prime \prime }(c)\right\} .
\label{[Sec2]-K11cc}
\end{equation}

\begin{remark}
\label{Remark21}The local character of the Taylor expansions means (\ref%
{[Sec2]-K01cc}) and (\ref{[Sec2]-K11cc}) hold for every $c\in \mathbb{R}$.
However, we are only interested in the case $c>0$ in order to study the
asymptotic behavior of sequences of polynomials orthogonal with respect to
the Sobolev-type inner product (\ref{[Sec1]-DicrLagSob}).
\end{remark}

The first technical step required for the proof of our main result is the
following lemma, concerning the asymptotic behavior as $n\rightarrow \infty $
of the above Laguerre kernels at $x=c$, $c\in \mathbb{R}_{+}$, that is,
within the oscillatory regime of the classical Laguerre orthogonal
polynomials.

\begin{lemma}
\label{[Sec2]-LEMMA-21} For every $c>0$, we have%
\begin{eqnarray*}
K_{n-1}(c,c) &\sim & \pi ^{-1}e^{c}c^{-\frac{1}{2}-\alpha }\,n^{1/2}, \\
K_{n-1}^{(0,1)}(c,c) &\sim &\pi ^{-1}e^{c}c^{-\frac{1}{2}-\alpha
}\,n^{1/2}, \\
K_{n-1}^{(1,1)}(c,c) &\sim &\frac{1}{3}\pi ^{-1}e^{c}c^{-\frac{3}{2}-\alpha
}\,n^{3/2}.
\end{eqnarray*}
\end{lemma}

\begin{proof}
Taking $p=1$ in (\ref{[Sec2]-Fejer}), we have $A_{0}(x)=1$ and $B_{0}(x)=0$.
Thus, we obtain the behavior of $\widehat{L}_{n}^{(\alpha )}(x)$ for $n$
large enough, when $x\in \mathbb{R}_{+}$
\begin{equation*}
\widehat{L}_{n}^{\alpha }(x)=(-1)^{n}\Gamma (n+1)\pi
^{-1/2}e^{x/2}x^{-\alpha /2-1/4}n^{\alpha /2-1/4}
\end{equation*}
\begin{equation*}
\cdot \cos \{2\left( nx\right) ^{1/2}-\alpha \pi /2-\pi /4\}\cdot (1+%
\mathcal{O}(n^{-1/2})).
\end{equation*}%
We can rewrite the above expression as%
\begin{equation}
\widehat{L}_{n}^{\alpha }(x)=(-1)^{n}\Gamma (n+1)n^{\frac{\alpha }{2}-\frac{1%
}{4}}\sigma ^{\alpha }(x)\cos \varphi _{n}^{\alpha }(x)(1+\mathcal{O}%
(n^{-1/2}))  \label{[Sec2]-Fejer-peq1}
\end{equation}%
where%
\begin{equation*}
\varphi _{n}^{\alpha }(x)=2(nx)^{1/2}-\frac{\alpha \pi }{2}-\frac{\pi }{4},
\end{equation*}%
and%
\begin{equation}
\sigma ^{\alpha }(x)=\pi ^{-1/2}e^{x/2}x^{-\alpha /2-1/4}
\label{[Sec2]-sigma}
\end{equation}%
being a function independent of $n$. Combining (\ref{[Sec2]-LagDer}) with (%
\ref{[Sec2]-Fejer-peq1}), we get
%for (\ref{[Sec2]-KLagxx})%
\begin{equation*}
K_{n-1}(c,c)\sim \frac{\Gamma (n+1)}{\Gamma (n+\alpha )}n^{\alpha }\Theta
_{n}(c;{\alpha }),
\end{equation*}%
where%
\begin{equation}
\Theta _{n}(c;{\alpha })=\sigma ^{\alpha }(c)\sigma ^{\alpha +1}(c)\left[
\cos \varphi _{n-1}^{\alpha +1}(c)\cos \varphi _{n-1}^{\alpha }(c)-\cos
\varphi _{n-2}^{\alpha +1}(c)\cos \varphi _{n}^{\alpha }(c)\right] .
\label{[Sec2]-THETA-0}
\end{equation}%
Let us examine the above expression. Using the trigonometric identity%
\begin{equation*}
\cos (a)\cos (b)=\frac{\cos (a+b)+\cos (a-b)}{2},
\end{equation*}
we have
\begin{align}
\frac{\Theta _{n}(c;{\alpha })}{\sigma ^{\alpha }(c)\sigma ^{\alpha +1}(c)}&
=\frac{1}{2}\cos \left( 4\sqrt{c\left( n-1\right) }-\pi \alpha -\pi \right)   \label{[Sec2]-THETA-1} \\
& -\frac{1}{2}\cos \left( 2\sqrt{nc}-\pi \alpha -\pi +2\sqrt{c\left(
n-2\right) }\right)   \notag \\
& -\frac{1}{2}\cos \left( 2\sqrt{c\left( n-2\right) }-2\sqrt{nc}-\frac{\pi }{%
2}\right) .  \notag
\end{align}%
The last term on the right hand side is
\begin{equation*}
\frac{-1}{2}\cos \left( 2\sqrt{(n-2)c}-2\sqrt{nc}-\frac{\pi }{2}\right) =%
\frac{1}{2}\sin \left( 2\sqrt{nc}-2\sqrt{(n-2)c}\right) ,
\end{equation*}%
which behaves with $n$ as follows%
\begin{align*}
& \lim_{n\rightarrow \infty }\frac{\sqrt{n}}{2}\sin \left( 2\sqrt{nc}-2\sqrt{%
(n-2)c}\right)  \\
& =\lim_{n\rightarrow \infty }\frac{\sqrt{n}}{2}\frac{\sin \left( 2\sqrt{nc}%
-2\sqrt{(n-2)c}\right) }{2\sqrt{nc}-2\sqrt{(n-2)c}}(2\sqrt{nc}-2\sqrt{(n-2)c}%
)=\sqrt{c},
\end{align*}%
and, therefore,%
\begin{equation}
\frac{1}{2}\sin \left( 2\sqrt{(n-2)c}-2\sqrt{nc}\right) \sim \sqrt{\frac{c}{n%
}}.  \label{[Sec2]-THETA-2}
\end{equation}%
Next we study%
\begin{equation}
\frac{1}{2}\cos \left( 4\sqrt{c\left( n-1\right) }-\pi \alpha -\pi \right) -%
\frac{1}{2}\cos \left( 2\sqrt{nc}-\pi \alpha -\pi +2\sqrt{c\left( n-2\right)
}\right)   \label{[Sec2]-THETA-3}
\end{equation}%
in (\ref{[Sec2]-THETA-1}). Using%
\begin{equation*}
\cos a-\cos b=-2\sin \left( \frac{a+b}{2}\right) \sin \left( \frac{a-b}{2}%
\right) ,
\end{equation*}%
we get that (\ref{[Sec2]-THETA-3}) becomes%
\begin{eqnarray*}
&&-\sin \left( \sqrt{cn}-\pi \alpha -\pi +2\sqrt{c\left( n-1\right) }+\sqrt{%
c\left( n-2\right) }\right)  \\
&&\qquad \qquad \qquad \qquad \cdot \sin \left( 2\sqrt{c\left( n-1\right) }-%
\sqrt{cn}-\sqrt{c\left( n-2\right) }\right)
\end{eqnarray*}%
where the first factor is bounded, and the second verifies%
\begin{equation}
\lim_{n\rightarrow \infty }\sqrt{n}\sin \left( 2\sqrt{c\left( n-1\right) }-%
\sqrt{cn}-\sqrt{c\left( n-2\right) }\right) =0.  \label{[Sec2]-THETA-4}
\end{equation}%
From (\ref{[Sec2]-THETA-2}) and (\ref{[Sec2]-THETA-4}), we conclude%
\begin{equation*}
\Theta _{n}(c;{\alpha })\sim \pi ^{-1}e^{c}c^{-\frac{1%
}{2}-\alpha}\,n^{-1/2}.
\end{equation*}
On the other hand, from the\ Stirling's formula for the Gamma function, we
deduce%
\begin{equation}
\frac{\Gamma (n+1)}{\Gamma (n+\alpha )}\sim n^{1-\alpha },
\label{[Sec2]-CocGammas}
\end{equation}%
under the above assumptions we get%
\begin{equation*}
K_{n-1}(c,c)\sim \pi ^{-1}e^{c}c^{-\frac{1}{2}-\alpha }\,n^{1/2},\quad c\in \mathbb{R}_{+}.
\end{equation*}%
\ Next, we can proceed as above and we obtain the asymptotic behavior given
in (\ref{[Sec2]-K01cc}). For $n$ large enough, we get%
\begin{equation}
K_{n-1}^{(0,1)}(c,c)\sim \frac{1}{2}\frac{\Gamma (n+1)}{\Gamma (n+\alpha )}%
n^{\alpha +\frac{1}{2}}\Psi_{n} (c;\alpha ),
\label{[Sec2]-K01cc-A}
\end{equation}%
where%
\begin{equation*}
\Psi_{n} (c;\alpha )=\sigma ^{\alpha }(c)\sigma ^{\alpha +2}(c)\left[ \cos
\varphi _{n}^{\alpha }(c)\cos \varphi _{n-3}^{\alpha +2}(c)-\cos \varphi
_{n-1}^{\alpha }(c)\cos \varphi _{n-2}^{\alpha +2}(c)\right] .
\end{equation*}%
The expression in square brackets can be rewritten as%
\begin{eqnarray*}
&&-\sin \left( \sqrt{cn}-\pi \alpha -\frac{3}{2}\pi +\sqrt{c\left(
n-1\right) }+\sqrt{c\left( n-2\right) }+\sqrt{c\left( n-3\right) }\right)  \\
&&\qquad \qquad \cdot \sin \left( \sqrt{cn}-\sqrt{c\left( n-1\right) }-\sqrt{%
c\left( n-2\right) }+\sqrt{c\left( n-3\right) }\right)  \\
&&-\sin \left( \pi +\sqrt{cn}+\sqrt{c\left( n-1\right) }-\sqrt{c\left(
n-2\right) }-\sqrt{c\left( n-3\right) }\right)  \\
&&\qquad \qquad \cdot \sin \left( \sqrt{cn}-\sqrt{c\left( n-1\right) }+\sqrt{%
c\left( n-2\right) }-\sqrt{c\left( n-3\right) }\right) ,
\end{eqnarray*}%
where%
\begin{equation*}
\begin{array}{c}
\lim_{n\rightarrow \infty }\left[ -n\sin \left( \sqrt{cn}-\pi \alpha -\frac{3%
}{2}\pi +\sqrt{c\left( n-1\right) }+\sqrt{c\left( n-2\right) }+\sqrt{c\left(
n-3\right) }\right) \right.  \\
\left. \cdot \sin \left( \sqrt{cn}-\sqrt{c\left( n-1\right) }-\sqrt{c\left(
n-2\right) }+\sqrt{c\left( n-3\right) }\right) \right] =0,%
\end{array}%
\end{equation*}%
and%
\begin{equation*}
\begin{array}{c}
\lim_{n\rightarrow \infty }\left[ -n\sin \left( \sqrt{cn}-\sqrt{c\left(
n-1\right) }+\sqrt{c\left( n-2\right) }-\sqrt{c\left( n-3\right) }\right)
\right.  \\
\left. \cdot \sin \left( \pi +\sqrt{cn}+\sqrt{c\left( n-1\right) }-\sqrt{%
c\left( n-2\right) }-\sqrt{c\left( n-3\right) }\right) \right] =2c.%
\end{array}%
\end{equation*}%
As a consequence, taking into account (\ref{[Sec2]-sigma}), we get%
\begin{equation*}
\Psi_{n} (c;\alpha )\sim \pi ^{-1}e^{c}c^{-\alpha -\frac{3}{2}}\cdot 2c\,n^{-1}.
\end{equation*}%
Replacing the above expression in (\ref{[Sec2]-K01cc-A}), and using again (%
\ref{[Sec2]-CocGammas}), we conclude%
\begin{equation*}
K_{n-1}^{(0,1)}(c,c)\sim \pi ^{-1}e^{c}c^{-\frac{1}{2}-\alpha
}\,n^{1/2}.
\end{equation*}%

Finally,%
\begin{equation}
K_{n-1}^{(1,1)}(c,c)\sim \frac{\Gamma (n+1)}{\Gamma (n+\alpha )}n^{\alpha
+1}\left( \frac{1}{3!}\Lambda _{1,n}(c;\alpha )+\frac{1}{2!}\Lambda _{2,n}(c;\alpha )\right),
\label{[Sec2]-Knm111}
\end{equation}%
where%
\begin{eqnarray}
\Lambda _{1,n}(c;\alpha ) &=&\sigma ^{\alpha }(c)\sigma ^{\alpha +3}(c)\left[
\cos \varphi _{n-3}^{\alpha +3}(c)\cos \varphi _{n-1}^{\alpha }(c)-\cos
\varphi _{n-4}^{\alpha +3}(c)\cos \varphi _{n}^{\alpha }(c)\right] , \label{expp1}\\
\Lambda _{2,n}(c;\alpha ) &=&\sigma ^{\alpha +1}(c)\sigma ^{\alpha +2}(c)\left[
\cos \varphi _{n-2}^{\alpha +2}(c)\cos \varphi _{n-2}^{\alpha +1}(c)-\cos
\varphi _{n-3}^{\alpha +2}(c)\cos \varphi _{n-1}^{\alpha +1}(c)\right] .\nonumber\\
\label{expp2}
\end{eqnarray}%
The two expressions in square brackets of \eqref{expp1} and \eqref{expp2} can be rewritten respectively, as follows
\begin{eqnarray*}
&&-\sin \left( \sqrt{nc}-\pi \alpha -2\pi +\sqrt{c\left( n-1\right) }+\sqrt{%
c\left( n-3\right) }+\sqrt{c\left( n-4\right) }\right)  \\
&&\qquad \qquad \cdot \sin \left( \sqrt{c\left( n-1\right) }-\sqrt{nc}+\sqrt{%
c\left( n-3\right) }-\sqrt{c\left( n-4\right) }\right)  \\
&&-\sin \left( \sqrt{c\left( n-3\right) }-\sqrt{nc}-\sqrt{c\left( n-1\right)
}-\frac{3}{2}\pi +\sqrt{c\left( n-4\right) }\right)  \\
&&\qquad \qquad \cdot \sin \left( \sqrt{nc}-\sqrt{c\left( n-1\right) }+\sqrt{%
c\left( n-3\right) }-\sqrt{c\left( n-4\right) }\right),
\end{eqnarray*}

\begin{eqnarray*}
&&-\sin \left( \sqrt{c\left( n-1\right) }-\pi \alpha -2\pi +2\sqrt{c\left(
n-2\right) }+\sqrt{c\left( n-3\right) }\right)  \\
&&\cdot \sin \left( 2\sqrt{c\left( n-2\right) }-\sqrt{c\left( n-1\right) }-%
\sqrt{c\left( n-3\right) }\right)  \\
&&-\frac{1}{2}\cos \left( 2\sqrt{c\left( n-3\right) }-2\sqrt{c\left(
n-1\right) }-\frac{1}{2}\pi \right) ,
\end{eqnarray*}%
where the terms of each sumand in the above expresions have the following behavior
\begin{equation*}
\begin{array}{c}
\lim_{n\rightarrow \infty }\left[ -n^{\frac{1}{2}}\sin \left( \sqrt{nc}-\pi
\alpha -2\pi +\sqrt{c\left( n-1\right) }+\sqrt{c\left( n-3\right) }+\sqrt{%
c\left( n-4\right) }\right) \right.  \\
\left. \cdot \sin \left( \sqrt{c\left( n-1\right) }-\sqrt{nc}+\sqrt{c\left(
n-3\right) }-\sqrt{c\left( n-4\right) }\right) \right] =0,%
\end{array}%
\end{equation*}%
\begin{equation*}
\begin{array}{c}
\lim_{n\rightarrow \infty }\left[ -n^{\frac{1}{2}}\sin \left( \sqrt{c\left(
n-3\right) }-\sqrt{nc}-\sqrt{c\left( n-1\right) }-\frac{3}{2}\pi +\sqrt{%
c\left( n-4\right) }\right) \right.  \\
\left. \cdot \sin \left( \sqrt{nc}-\sqrt{c\left( n-1\right) }+\sqrt{c\left(
n-3\right) }-\sqrt{c\left( n-4\right) }\right) \right] =-\sqrt{c},%
\end{array}%
\end{equation*}%
\begin{equation*}
\begin{array}{c}
\lim_{n\rightarrow \infty }\left[ -n^{\frac{1}{2}}\sin \left( \sqrt{c\left(
n-1\right) }-\pi \alpha -2\pi +2\sqrt{c\left( n-2\right) }+\sqrt{c\left(
n-3\right) }\right) \right.  \\
\left. \cdot \sin \left( 2\sqrt{c\left( n-2\right) }-\sqrt{c\left(
n-1\right) }-\sqrt{c\left( n-3\right) }\right) \right] =0,%
\end{array}%
\end{equation*}%
and%
\begin{equation*}
\lim\limits_{n\rightarrow \infty }\left( n^{\frac{1}{2}}\left( -\frac{1}{2}%
\cos \left( 2\sqrt{c\left( n-3\right) }-2\sqrt{c\left( n-1\right) }-\frac{1}{%
2}\pi \right) \right) \right) =\sqrt{c}.
\end{equation*}%

Hence, using again (\ref{[Sec2]-sigma}), we have%
\begin{eqnarray*}
\Lambda _{1,n}(c;\alpha ) &\sim &-\pi ^{-1}e^{c}c^{-2-\alpha }\sqrt{c}\,n^{-1/2}, \\
\Lambda _{2,n}(c;\alpha ) &\sim &\pi ^{-1}e^{c}c^{-2-\alpha }\sqrt{c}\,n^{-1/2}.
\end{eqnarray*}%
Therefore%
\begin{equation*}
\left( \frac{1}{3!}\Lambda _{1,n}(c;\alpha )+\frac{1}{2!}\Lambda _{2,n}(c;\alpha
)\right) \sim\frac{1}{3}\pi ^{-1}e^{c}c^{-\alpha -\frac{3}{2}}\,n^{-1/2}.
\end{equation*}%
Replacing in (\ref{[Sec2]-Knm111}) we conclude,
\begin{equation*}
K_{n-1}^{(1,1)}(c,c)\sim \frac{1}{3}\pi ^{-1}e^{c}c^{-\frac{3}{2}-\alpha
}\,n^{3/2}.
\end{equation*}

\end{proof}

%%%%%%%%%%%%%%%%%%%%%%%%%%%%%%%%%%%%%%%%%%%%%%%%%%%%%%%%%%%%%%%%%%%%%%%%%%%%%%%%%%

\section{Outer Relative Asymptotics for $c$ on $\mathbb{R}_{+}$}

\label{[Section-3]-OuterAs}

%%%%%%%%%%%%%%%%%%%%%%%%%%%%%%%%%%%%%%%%%%%%%%%%%%%%%%%%%%%%%%%%%%%%%%%%%%%%%%%%%%

The main result of this section will be the outer relative asymptotics for
the Laguerre-Sobolev type polynomials $\widehat{S}_{n}^{M,N}(x)$, orthogonal
with respect to (\ref{[Sec1]-DicrLagSob}), when $c\in \mathbb{R}_{+}$. The
proof will naturally falls in several parts, which will be established
through an appropriate sequence of Lemmas.

First, we will present a well known expansion of the monic polynomials $%
\widehat{S}_{n}^{M,N}(x)$\ in terms of classical Laguerre polynomials $%
\widehat{L}_{n}^{\alpha }(x)$. The most common way to represent the
Laguerre-Sobolev type orthogonal polynomials $\widehat{S}_{n}^{M,N}(x)$ is
using the Laguerre kernel and its derivatives as follows (see \cite{MR90} and Theorem 5.1 in
\citep{H12}).
\begin{equation}
(x-c)^{2}\widehat{S}_{n}^{M,N}(x)=A(n;x)\widehat{L}_{n}^{\alpha }(x)+B(n;x)%
\widehat{L}_{n-1}^{\alpha }(x),  \label{[Sec3]-ConnFormS}
\end{equation}%
where%
\begin{equation}
\begin{array}{l}
A(n;x)=(x-c)^{2}+(x-c)A_{1}(n;c)+A_{0}(n;c), \\
B(n;x)=(x-c)B_{1}(n;c)+B_{0}(n;c),%
\end{array}
\label{[Sec3]-ABxn}
\end{equation}%
with%
\begin{equation}
\begin{array}{l}
A_{1}(n;c)=-\frac{M\widehat{S}_{n}^{M,N}(c)\widehat{L}_{n-1}^{\alpha }(c)}{||%
\widehat{L}_{n-1}^{\alpha }||_{\alpha }^{2}}-\frac{N[\widehat{S}%
_{n}^{M,N}]^{\prime }(c)[\widehat{L}_{n-1}^{\alpha }]^{\prime }(c)}{||%
\widehat{L}_{n-1}^{\alpha }||_{\alpha }^{2}}, \\
A_{0}(n;c)=-\frac{N[\widehat{S}_{n}^{M,N}]^{\prime }(c)\widehat{L}%
_{n-1}^{\alpha }(c)}{||\widehat{L}_{n-1}^{\alpha }||_{\alpha }^{2}}, \\
B_{1}(n;c)=\frac{M\widehat{S}_{n}^{M,N}(c)\widehat{L}_{n}^{\alpha }(c)}{||%
\widehat{L}_{n-1}^{\alpha }||_{\alpha }^{2}}+\frac{N[\widehat{S}%
_{n}^{M,N}]^{\prime }(c)[\widehat{L}_{n}^{\alpha }]^{\prime }(c)}{||\widehat{%
L}_{n-1}^{\alpha }||_{\alpha }^{2}}, \\
B_{0}(n;c)=\frac{N[\widehat{S}_{n}^{M,N}]^{\prime }(c)\widehat{L}%
_{n}^{\alpha }(c)}{||\widehat{L}_{n-1}^{\alpha }||_{\alpha }^{2}}.%
\end{array}
\label{[Sec3]-CoefA0A1B0B1}
\end{equation}%
Notice that
\begin{eqnarray}
\widehat{S}_{n}^{M,N}(c) &=&\frac{%
\begin{vmatrix}
\widehat{L}_{n}^{\alpha }(c) & NK_{n-1}^{(0,1)}(c,c) \\
\lbrack \widehat{L}_{n}^{\alpha }]^{\prime }(c) & 1+NK_{n-1}^{(1,1)}(c,c)%
\end{vmatrix}%
}{%
\begin{vmatrix}
1+MK_{n-1}(c,c) & NK_{n-1}^{(0,1)}(c,c) \\
MK_{n-1}^{(1,0)}(c,c) & 1+NK_{n-1}^{(1,1)}(c,c)%
\end{vmatrix}%
},  \label{[Sec3]-Sn-Det} \\
\lbrack \widehat{S}_{n}^{M,N}]^{\prime }(c) &=&\frac{%
\begin{vmatrix}
1+MK_{n-1}(c,c) & \widehat{L}_{n}^{\alpha }(c) \\
MK_{n-1}^{(1,0)}(c,c) & [\widehat{L}_{n}^{\alpha }]^{\prime }(c)%
\end{vmatrix}%
}{%
\begin{vmatrix}
1+MK_{n-1}(c,c) & NK_{n-1}^{(0,1)}(c,c) \\
MK_{n-1}^{(1,0)}(c,c) & 1+NK_{n-1}^{(1,1)}(c,c)%
\end{vmatrix}%
}.  \label{[Sec3]-Sprimn-Det}
\end{eqnarray}

We will analyze the polynomial coefficients in the above expansion in order
to obtain the desired results. If we replace \eqref{[Sec3]-Sn-Det} and \eqref{[Sec3]-Sprimn-Det} in \eqref{[Sec3]-CoefA0A1B0B1}, we obtain
\begin{equation*}
\begin{array}{l}
A_{1}(n;c)=\frac{-M\widehat{L}_{n-1}^{\alpha }(c)\widehat{L}_{n}^{\alpha }(c)-MN\widehat{L}_{n-1}^{\alpha }(c)\widehat{L}_{n}^{\alpha }(c)K_{n-1}^{(1,1)}(c,c)+MNn \widehat{L}_{n-1}^{\alpha }(c)\widehat{L}_{n-1}^{\alpha+1 }(c)K_{n-1}^{(0,1)}(c,c)}{||%
\widehat{L}_{n-1}^{\alpha }||_{\alpha }^{2} \left(1+MK_{n-1}(c,c)+NK_{n-1}^{(1,1)}(c,c)+MNK_{n-1}(c,c)K_{n-1}^{(1,1)}(c,c)-MNK_{n-1}^{(0,1)}(c,c)K_{n-1}^{(1,0)}(c,c) \right)}\\
+\frac{\left(-Nn^2\widehat{L}_{n-2}^{\alpha+1}(c)\widehat{L}_{n-1}^{\alpha+1}(c)-MNn^2\widehat{L}_{n-2}^{\alpha+1 }(c)\widehat{L}_{n-1}^{\alpha+1 }(c)K_{n-1}(c,c)+MNn \widehat{L}_{n-2}^{\alpha+1 }(c)\widehat{L}_{n}^{\alpha }(c)K_{n-1}^{(1,0)}(c,c)\right)}{||%
\widehat{L}_{n-1}^{\alpha }||_{\alpha }^{2} \left(1+MK_{n-1}(c,c)+NK_{n-1}^{(1,1)}(c,c)+MNK_{n-1}(c,c)K_{n-1}^{(1,1)}(c,c)-MNK_{n-1}^{(0,1)}(c,c)K_{n-1}^{(1,0)}(c,c) \right)}, \\ \\
A_{0}(n;c)=\frac{-Nn\widehat{L}_{n-1}^{\alpha }(c)\widehat{L}_{n-1}^{\alpha+1 }(c)-MNn\widehat{L}_{n-1}^{\alpha }(c)\widehat{L}_{n-1}^{\alpha+1 }(c)K_{n-1}(c,c)+MN \widehat{L}_{n-1}^{\alpha }(c)\widehat{L}_{n}^{\alpha}(c)K_{n-1}^{(1,0)}(c,c)}{||%
\widehat{L}_{n-1}^{\alpha }||_{\alpha }^{2} \left(1+MK_{n-1}(c,c)+NK_{n-1}^{(1,1)}(c,c)+MNK_{n-1}(c,c)K_{n-1}^{(1,1)}(c,c)-MNK_{n-1}^{(0,1)}(c,c)K_{n-1}^{(1,0)}(c,c) \right)}, \\ \\
B_{1}(n;c)=\frac{M\widehat{L}_{n}^{\alpha }(c)\widehat{L}_{n}^{\alpha }(c)+MN\widehat{L}_{n}^{\alpha }(c)\widehat{L}_{n}^{\alpha }(c)K_{n-1}^{(1,1)}(c,c)-MNn \widehat{L}_{n}^{\alpha }(c)\widehat{L}_{n-1}^{\alpha+1}(c)K_{n-1}^{(0,1)}(c,c)}{||%
\widehat{L}_{n-1}^{\alpha }||_{\alpha }^{2} \left(1+MK_{n-1}(c,c)+NK_{n-1}^{(1,1)}(c,c)+MNK_{n-1}(c,c)K_{n-1}^{(1,1)}(c,c)-MNK_{n-1}^{(0,1)}(c,c)K_{n-1}^{(1,0)}(c,c) \right)}\\
+\frac{Nn^2\widehat{L}_{n-1}^{\alpha+1}(c)\widehat{L}_{n-1}^{\alpha+1}(c)+MNn^2\widehat{L}_{n-1}^{\alpha+1}(c)\widehat{L}_{n-1}^{\alpha+1}(c)K_{n-1}(c,c)-MNn \widehat{L}_{n-1}^{\alpha+1 }(c)\widehat{L}_{n}^{\alpha }(c)K_{n-1}^{(1,0)}(c,c)}{||%
\widehat{L}_{n-1}^{\alpha }||_{\alpha }^{2} \left(1+MK_{n-1}(c,c)+NK_{n-1}^{(1,1)}(c,c)+MNK_{n-1}(c,c)K_{n-1}^{(1,1)}(c,c)-MNK_{n-1}^{(0,1)}(c,c)K_{n-1}^{(1,0)}(c,c) \right)}, \\ \\
B_{0}(n;c)=\frac{Nn\widehat{L}_{n}^{\alpha}(c)\widehat{L}_{n-1}^{\alpha+1}(c)+MNn\widehat{L}_{n}^{\alpha}(c)\widehat{L}_{n-1}^{\alpha+1 }(c)K_{n-1}(c,c)-MN \widehat{L}_{n}^{\alpha}(c)\widehat{L}_{n}^{\alpha}(c)K_{n-1}^{(1,0)}(c,c)}{||%
\widehat{L}_{n-1}^{\alpha }||_{\alpha }^{2} \left(1+MK_{n-1}(c,c)+NK_{n-1}^{(1,1)}(c,c)+MNK_{n-1}(c,c)K_{n-1}^{(1,1)}(c,c)-MNK_{n-1}^{(0,1)}(c,c)K_{n-1}^{(1,0)}(c,c) \right)}.%
\end{array}
\end{equation*}

Using \eqref{[Sec2]-Fejer-peq1} and the estimates in Lemma \ref{[Sec2]-LEMMA-21}, we can compute the asymptotic behavior of the previous expressions as follows.

\begin{multline}\label{[Sec3]estimA0A1B0B1}
A_{1}(n;c)\sim \frac{1}{Nc \sigma^{\alpha+1}(c)\sigma^{\alpha+3}(c)} n^{-3/2} \cos \varphi_{n-1}^{\alpha}(c)\cos \varphi_{n}^{\alpha}(c)
+  \cos \varphi_{n-1}^{\alpha}(c)\cos \varphi_{n}^{\alpha}(c) \\
 +2\sqrt{c} n^{-1/2} \cos \varphi_{n-1}^{\alpha}(c)\cos \varphi_{n-1}^{\alpha+1}(c)
+\frac{1}{M \sigma^{\alpha}(c)\sigma^{\alpha}(c)} n^{-1/2} \cos \varphi_{n-2}^{\alpha+1}(c)\cos \varphi_{n-1}^{\alpha+1}(c) \\
+ \cos \varphi_{n-2}^{\alpha+1}(c)\cos \varphi_{n-1}^{\alpha+1}(c)
+2 n^{-1/2}\cos \varphi_{n-2}^{\alpha+1}(c)\cos \varphi_{n}^{\alpha}(c),\\ \\
A_{0}(n;c) \sim\frac{-1}{M c\sigma^{\alpha}(c)\sigma^{\alpha+3}(c)} n^{-1} \cos \varphi_{n-1}^{\alpha}(c)\cos \varphi_{n-1}^{\alpha+1}(c)
-c^{1/2} n^{-1/2} \cos \varphi_{n-1}^{\alpha}(c)\cos \varphi_{n-1}^{\alpha+1}(c) \\
-2n^{-1}\cos \varphi_{n-1}^{\alpha}(c)\cos \varphi_{n}^{\alpha}(c), \\ \\
B_{1}(n;c) \sim \frac{1}{Nc \sigma^{\alpha+1}(c)\sigma^{\alpha+3}(c)} n^{-1/2} \cos \varphi_{n}^{\alpha}(c)\cos \varphi_{n}^{\alpha}(c)
+n\cos \varphi_{n}^{\alpha}(c)\cos \varphi_{n}^{\alpha}(c) \\
 +2\sqrt{c}n^{1/2} \cos \varphi_{n}^{\alpha}(c)\cos \varphi_{n-1}^{\alpha+1}(c)
+\frac{1}{M \sigma^{\alpha}(c)\sigma^{\alpha}(c)} n^{1/2} \cos \varphi_{n-1}^{\alpha+1}(c)\cos \varphi_{n-1}^{\alpha+1}(c) \\
+ n\cos \varphi_{n-1}^{\alpha+1}(c)\cos \varphi_{n-1}^{\alpha+1}(c)
+2  n^{1/2} \cos \varphi_{n-1}^{\alpha+1}(c)\cos \varphi_{n}^{\alpha}(c), \\ \\
B_{0}(n;c)\sim \frac{-1}{M c\sigma^{\alpha}(c)\sigma^{\alpha+3}(c)}  \cos \varphi_{n}^{\alpha}(c)\cos \varphi_{n-1}^{\alpha+1}(c)
-c^{1/2} n^{1/2} \cos \varphi_{n}^{\alpha}(c)\cos \varphi_{n-1}^{\alpha+1}(c) \\
-2\cos \varphi_{n}^{\alpha}(c)\cos \varphi_{n}^{\alpha}(c) .%
\end{multline}

Due to the oscillatory behaviour of the cosines functions appearing in the preceding formulas, there are no real numbers $\beta_0$ and $\beta_1$ such that
\begin{align*}
A_0(n;c) &\sim C_0 n^{\beta_0}, \\
B_0(n;c) &\sim C_1 n^{\beta_1},
\end{align*}
for some $C_0$ and $C_1.$

\medskip
However, we can describe the asymptotic behaviour of our coefficients functions in the following way:

\begin{proposition}\label{[Sec3]propA0A1B0B1}
Let $A_0(n;c), A_1(n;c), B_0(n;c)$ and $B_1(n;c)$ the functions defined by \eqref{[Sec3]-CoefA0A1B0B1}. Then, we have
\begin{equation*}\label{[Sec3]behavA0A1B0B1}
\begin{array}{lll}
A_1(n;c)\sim 1, &  &\displaystyle \lim_{n\rightarrow \infty} n^{\beta} A_0(n;c) =
\begin{cases}
0 &\quad \text{if  } \beta<\frac{1}{2}, \\
\nexists&\quad \text{if  } \beta\geq\frac{1}{2},
\end{cases} \\ \\
B_1(n;c)\sim n, &  &\displaystyle \lim_{n\rightarrow \infty} n^{\beta} B_0(n;c) =
\begin{cases}
0 &\quad \text{if  } \beta<-\frac{1}{2}, \\
\nexists&\quad \text{if  } \beta\geq -\frac{1}{2}.
\end{cases}%
\end{array}%
\end{equation*}%

\end{proposition}

\begin{proof}
The asymptotic behaviour of $A_0 (n;c)$ and $B_0(n;c)$ is an immediate consequence of the estimates in \eqref{[Sec3]estimA0A1B0B1}.

In order to obtain the asymptotics for $A_1(n;c)$ and $B_1(n;c)$, we joint up the terms
\begin{align*}
&\cos \varphi_{n-1}^{\alpha}(c)\cos \varphi_{n}^{\alpha}(c)+ \cos \varphi_{n-2}^{\alpha+1}(x)\cos \varphi_{n-1}^{\alpha+1}(c)= \\
&\cos \left( 2\sqrt{c(n-1)}+\sqrt{cn}+\sqrt{c(n-2)}-\alpha\pi-\pi \right) \cos \left(\sqrt{cn}-\sqrt{c(n-2)}+\frac{\pi}{2}\right) \\
&+\frac{1}{2} \cos \left(2\sqrt{c(n-1)}-2\sqrt{cn}\right) +\frac{1}{2} \cos \left(2\sqrt{c(n-2)}-2\sqrt{c(n-1)}\right),
\end{align*}
and
\begin{align*}
&\cos \varphi_{n}^{\alpha}(c)\cos \varphi_{n}^{\alpha}(c) +\cos \varphi_{n-1}^{\alpha+1}(c)\cos \varphi_{n-1}^{\alpha+1}(c)= \\
&\cos \left( 2\sqrt{cn}+2\sqrt{c(n-1)}-\alpha\pi-\pi \right) \cos \left(2\sqrt{cn}-2\sqrt{c(n-1)}+\frac{\pi}{2}\right)+1 .
\end{align*}

Taking into account that the previous expressions tend to $1$ as $n$ tends to infinity, we obtain the desired result.

\end{proof}

We can now formulate our main result.

\begin{theorem}
\label{[Sec3]-THEOREM-31} The outer relative asymptotics for Laguerre
Sobolev-type polynomials $\widehat{S}_{n}^{M,N}(x)$, orthogonal with respect to the discrete Sobolev
inner product (\ref{[Sec1]-DicrLagSob}), is
\begin{equation*}
\lim\limits_{n\rightarrow \infty }\frac{\widehat{S}_{n}^{M,N}(x)}{\widehat{L}%
_{n}^{\alpha }(x)}=1,
\end{equation*}%
uniformly on compact subsets of $\mathbb{C}\setminus \mathbb{R}_{+}$.
\end{theorem}

\begin{proof}
Replacing (\ref{[Sec3]-ABxn}) in (\ref{[Sec3]-ConnFormS})
\begin{equation}
\frac{\widehat{S}_{n}^{M,N}(x)}{\widehat{L}_{n}^{\alpha }(x)}=\left\{ 1+%
\frac{A_{1}(n;c)}{(x-c)}+\frac{A_{0}(n;c)}{(x-c)^{2}}\right\} +\left\{ \frac{%
B_{1}(n;c)}{(x-c)}+\frac{B_{0}(n;c)}{(x-c)^{2}}\right\} \frac{\widehat{L}%
_{n-1}^{\alpha }(x)}{\widehat{L}_{n}^{\alpha }(x)},  \label{[Sec3]-ConnFo2}
\end{equation}%
From the Perron's formula (\ref{[Sec2]-Perron}) (for more
details we refer the reader to \citep{DeHM13}) we get
\begin{equation*}
\frac{L_{n-1}^{(\alpha )}(x)}{L_{n}^{(\alpha )}(x)}=1-\frac{\sqrt{-x}}{\sqrt{%
n}}+\mathcal{O}(n^{-1}).
\end{equation*}%
For monic polynomials (\ref{[Sec2]-(-1)nL}) the above relation becomes%
\begin{equation}
\frac{\widehat{L}_{n-1}^{\alpha }(x)}{\widehat{L}_{n}^{\alpha }(x)}=\frac{-1%
}{n}\left( 1-\frac{\sqrt{-x}}{\sqrt{n}}+\mathcal{O}(n^{-1})\right) .
\label{[Sec3]-CocLagMon}
\end{equation}%
By using (\ref{[Sec3]-CocLagMon}) we can rewrite (\ref{[Sec3]-ConnFo2}) as%
\begin{equation*}
\frac{\widehat{S}_{n}^{M,N}(x)}{\widehat{L}_{n}^{\alpha }(x)}\sim \left\{ 1+%
\frac{A_{1}(n;c)}{(x-c)}+\frac{A_{0}(n;c)}{(x-c)^{2}}\right\} -\left\{ \frac{%
\frac{B_{1}(n;c)}{n}}{(x-c)}+\frac{\frac{B_{0}(n;c)}{n}}{(x-c)^{2}}\right\}.
\end{equation*}%
Then, in order to conclude our proof, we only need to check that
\begin{eqnarray}
\lim_{n\rightarrow\infty} \left( A_1(n;c) -\frac{B_1(n;c)}{n} \right)&=& 0, \label{ec1}\\
\lim_{n\rightarrow\infty} \left( A_0(n;c) -\frac{B_0(n;c)}{n} \right)&=& 0.\label{ec2}
\end{eqnarray}
By applying Proposition \ref{[Sec3]propA0A1B0B1}, we obtain \eqref{ec1}.
From \eqref{[Sec3]estimA0A1B0B1}, we get
\begin{align*}
A_0(n;c) -\frac{B_0(n;c)}{n}&\sim \frac{-1}{Mc \sigma^{\alpha}(c)\sigma^{\alpha+3}(c)} n^{-1}\left( \cos \varphi_{n-1}^{\alpha}(c)\cos \varphi_{n-1}^{\alpha+1}(c)
                              -\cos \varphi_{n}^{\alpha}(c)\cos \varphi_{n-1}^{\alpha+1}(c)\right) \\
                              &-c^{1/2} n^{-1/2}\left( \cos \varphi_{n-1}^{\alpha}(c)\cos \varphi_{n-1}^{\alpha+1}(c)-\cos \varphi_{n}^{\alpha}(c)\cos \varphi_{n-1}^{\alpha+1}(c)\right) \\
                              &2n^{-1}\left(\cos \varphi_{n-1}^{\alpha}(c)\cos \varphi_{n}^{\alpha}(c)-\cos \varphi_{n}^{\alpha}(c)\cos \varphi_{n}^{\alpha}(c)\right).
\end{align*}
Since this expression tend to zero as $n$ tends to infinity, then \eqref{ec2} hold.

\end{proof}

%%%%%%%%%%%%%%%%%%%%%%%%%%%%%%%%%%%%%%%%%%%%%%%%%%%%%%%%%%%%%%%%%%%%%%%%%%%%%%%%%%

\section{The five-term recurrence relation}

\label{[Section-4]-5TRR}

%%%%%%%%%%%%%%%%%%%%%%%%%%%%%%%%%%%%%%%%%%%%%%%%%%%%%%%%%%%%%%%%%%%%%%%%%%%%%%%%%%

This section is focused on the five-term recurrence relation
that the sequence of discrete Laguerre--Sobolev orthogonal
polynomials $\{\widehat{S}_{n}^{M,N}(x)\}_{n\geq 0}$ satisfies. Next, we will estimate
the coefficients of such a recurrence relation for $n$ large enough and $%
c\in \mathbb{R}_{+}$. To this end, we will use the remarkable fact, which is
a straightforward consequence of (\ref{[Sec1]-DicrLagSob}), that the
multiplication operator by $(x-c)^{2}$ is a symmetric operator with respect
to such a discrete Sobolev inner product. Indeed, for any $f(x),g(x)\in
\mathbb{P}$
\begin{equation}
\langle (x-c)^{2}f(x),g(x)\rangle _{S}=\langle f(x),(x-c)^{2}g(x)\rangle
_{S}.  \label{[Sec4]-SymmInn}
\end{equation}

%If we denote $q(x)=(x-c)^{2}g(x)$ in the above inner product, then
%\begin{equation*}
%\lbrack (x-c)^{2}g(x)]|_{x=c}=[(x-c)^{2}g(x)]^{\prime }|_{x=c}=0.
%\end{equation*}%
Notice that%
\begin{equation}
\langle (x-c)^{2}f(x),g(x)\rangle _{S}=\langle f(x),g(x)\rangle _{\lbrack
2]}.  \label{[Sec4]-Property1}
\end{equation}%
An equivalent formulation of (\ref{[Sec4]-Property1}) is%
\begin{equation}
\langle (x-c)^{2}f(x),g(x)\rangle _{S}=\langle (x-c)^{2}f(x),g(x)\rangle
_{\alpha }.  \label{[Sec4]-Property2}
\end{equation}

We will need some preliminary results that will be stated as
Lemmas \ref{[Sec4]-LEMMA-41}, and \ref{[Sec4]-LEMMA-42}.

\begin{lemma}
\label{[Sec4]-LEMMA-41}For every $n\geq 1$ and initial conditions $\widehat{L%
}_{-1}^{\alpha }(x)=0$, $\widehat{L}_{0}^{\alpha }(x)=1$, $\widehat{L}%
_{1}^{\alpha }(x)=x-(\alpha +1)$, the connection formula (\ref%
{[Sec3]-ConnFormS}) reads as%
\begin{equation*}
(x-c)^{2}\widehat{S}_{n}^{M,N}(x)=
\end{equation*}%
\begin{equation*}
\widehat{L}_{n+2}^{\alpha }(x)+\tilde{b}_{n}\widehat{L}_{n+1}^{\alpha }(x)+%
\tilde{c}_{n}\widehat{L}_{n}^{\alpha }(x)+\tilde{d}_{n}\widehat{L}%
_{n-1}^{\alpha }(x)+\tilde{e}_{n}\widehat{L}_{n-2}^{\alpha }(x),
\end{equation*}%
where%
\begin{equation*}
\tilde{b}_{n}=\beta _{n+1}+\beta _{n}-2c+A_{1}(n;c)\sim 4n,
\end{equation*}%
\begin{equation*}
\tilde{c}_{n}=\gamma _{n+1}+\gamma _{n}+(\beta _{n}-c)^{2}+A_{1}(n;c)\left[
\beta _{n}-c\right] +A_{0}(n;c) +B_{1}(n;c)\sim 6n^{2},
\end{equation*}%
\begin{equation*}
\tilde{d}_{n}=\gamma _{n}(\beta _{n}+\beta _{n-1}-2c)+\gamma
_{n}A_{1}(n;c)+(\beta _{n-1}-c)B_{1}(n;c)+B_{0}(n;c)\sim 4n^{3},
\end{equation*}%
\begin{equation*}
\tilde{e}_{n}=\gamma _{n}\gamma _{n-1}+\gamma _{n-1}B_{1}(n;c)\sim n^{4}.
\end{equation*}
\end{lemma}

\begin{proof}
We begin with the expression%
\begin{equation*}
(x-c)^{2}\widehat{L}_{n}^{\alpha }(x)=
\end{equation*}%
\begin{equation}
\widehat{L}_{n+2}^{\alpha }(x)+b_{n}\widehat{L}_{n+1}^{\alpha }(x)+c_{n}%
\widehat{L}_{n}^{\alpha }(x)+d_{n}\widehat{L}_{n-1}^{\alpha }(x)+e_{n}%
\widehat{L}_{n-2}^{\alpha }(x),  \label{[Sec4]-Lem41-01}
\end{equation}%
where%
\begin{equation*}
\begin{array}{ll}
b_{n}=\beta _{n+1}+\beta _{n}-2c\sim 4n, & c_{n}=\gamma _{n+1}+\gamma
_{n}+(\beta _{n}-c)^{2}\sim 6n^{2}, \\
d_{n}=\gamma _{n}(\beta _{n}+\beta _{n-1}-2c)\sim 4n^{3}, & e_{n}=\gamma
_{n}\gamma _{n-1}\sim n^{4},%
\end{array}%
\end{equation*}%
according to (\ref{[Sec2]-3TRRLag}) and the definition of $\beta _{n}$ and $%
\gamma _{n}$ in (\ref{[Sec2]-3TRRLag}).

From the expression of $A(n;x)$ in (\ref{[Sec3]-ABxn}), the next step is to
expand the polynomial
$\left[ A_{1}(n;x)(x-c)+A_{0}(n;x)\right] \widehat{L}_{n}^{\alpha }(x)$ in
terms of $\{\widehat{L}_{n}^{\alpha }\}_{n\geq 0}$. Indeed, from (\ref%
{[Sec2]-3TRRLag})%
\begin{equation*}
\left[ A_{1}(n;x)(x-c)+A_{0}(n;x)\right] \widehat{L}_{n}^{\alpha }(x)=
\end{equation*}%
\begin{equation*}
A_{1}(n;x)\widehat{L}_{n+1}^{\alpha }(x)+\left[ (\beta
_{n}-c)A_{1}(n;x)+A_{0}(n;x)\right] \widehat{L}_{n}^{\alpha
}(x)+A_{1}(n;x)\gamma _{n}\widehat{L}_{n-1}^{\alpha }(x).
\end{equation*}%
Adding these coefficients to those of (\ref{[Sec4]-Lem41-01}), we obtain%
\begin{equation*}
A(n;x)\widehat{L}_{n}^{\alpha }(x)=\widehat{L}_{n+2}^{\alpha }(x)+\bar{b}_{n}%
\widehat{L}_{n+1}^{\alpha }(x)+\bar{c}_{n}\widehat{L}_{n}^{\alpha }(x)+\bar{d%
}_{n}\widehat{L}_{n-1}^{\alpha }(x)+\bar{e}_{n}\widehat{L}_{n-2}^{\alpha
}(x),
\end{equation*}%
with%
\begin{equation*}
\begin{array}{ll}
\bar{b}_{n}=b_{n}+A_{1}(n;c)\sim 4n, & \bar{c}_{n}=c_{n}+A_{1}(n;c)\left(
\beta _{n}-c\right)+A_{0}(n;c) \sim 6n^{2}, \\
\bar{d}_{n}=d_{n}+\gamma _{n}A_{1}(n;c)\sim 4n^{3}, & \bar{e}_{n}=e_{n}\sim
n^{4},%
\end{array}%
\end{equation*}%
where we have used Proposition \ref{[Sec3]propA0A1B0B1}.
In a similar way, for $B(n;x)$ in (\ref{[Sec3]-ABxn}) we get%
\begin{equation*}
B(n;x)\widehat{L}_{n-1}^{\alpha }(x)=\breve{c}_{n}\widehat{L}_{n}^{\alpha
}(x)+\breve{d}_{n}\widehat{L}_{n-1}^{\alpha }(x)+\breve{e}_{n}\widehat{L}%
_{n-2}^{\alpha }(x),
\end{equation*}%
where%
\begin{equation*}
\begin{array}{l}
\breve{c}_{n}=B_{1}(n;c)\sim n, \\
\breve{d}_{n}=(\beta _{n-1}-c)B_{1}(n;c)+B_{0}(n;c)\sim  2n^{2}, \\
\breve{e}_{n}=\gamma _{n-1}B_{1}(n;c)\sim n^{3}.%
\end{array}%
\end{equation*}%
As a conclusion,%
\begin{eqnarray*}
(x-c)^{2}\widehat{S}_{n}^{M,N}(x) &=&A(n;x)\widehat{L}_{n}^{\alpha
}(x)+B(n;x)\widehat{L}_{n-1}^{\alpha }(x) \\
&=&\widehat{L}_{n+2}^{\alpha }(x)+\bar{b}_{n}\widehat{L}_{n+1}^{\alpha }(x)+(%
\bar{c}_{n}+\breve{c}_{n})\widehat{L}_{n}^{\alpha }(x) \\
&&+(\bar{d}_{n}+\breve{d}_{n})\widehat{L}_{n-1}^{\alpha }(x)+(\bar{e}_{n}+%
\breve{e}_{n})\widehat{L}_{n-2}^{\alpha }(x).
\end{eqnarray*}%
This completes the proof.
\end{proof}

\begin{lemma}
\label{[Sec4]-LEMMA-42}For every $\alpha >-1$, $n\geq 1$, and $c\in \mathbb{R%
}_{+}$ the norm of the Laguerre-Sobolev type polynomials $\widehat{S}%
_{n}^{M,N}$, orthogonal with respect to (\ref{[Sec1]-DicrLagSob}) is%
\begin{equation*}
||\widehat{S}_{n}^{M,N}||_{S}^{2}=||\widehat{L}_{n}^{\alpha }||_{\alpha
}^{2}+B_{1}(n;c)||\widehat{L}_{n-1}^{\alpha }||_{\alpha }^{2}\sim \Gamma (n+1)\Gamma (n+\alpha +1).
\end{equation*}%
where $B_{1}(n;c)$ is the polynomial coefficient defined in \eqref{[Sec3]-CoefA0A1B0B1}.
\end{lemma}

\begin{proof}
First, we notice that%
\begin{equation*}
||\widehat{S}_{n}^{M,N}||_{S}^{2}=\langle \widehat{S}_{n}^{M,N}(x),(x-c)^{2}%
\widehat{\Pi }_{n-2}(x)\rangle _{S},
\end{equation*}%
for every monic polynomial $\widehat{\Pi }_{n-2}$ of degree $n-2$ . From (%
\ref{[Sec4]-Property2})%
\begin{eqnarray*}
\langle \widehat{S}_{n}^{M,N}(x),(x-c)^{2}\widehat{\Pi }_{n-2}(x)\rangle
_{S} &=&\langle (x-c)^{2}\widehat{S}_{n}^{M,N}(x),\widehat{\Pi }%
_{n-2}(x)\rangle _{S} \\
&=&\langle (x-c)^{2}\widehat{S}_{n}^{M,N}(x),\widehat{\Pi }_{n-2}(x)\rangle
_{\alpha }.
\end{eqnarray*}%
Next we use the connection formula (\ref{[Sec3]-ConnFormS}). Taking into
account that $A(n;x)$ is a monic quadratic polynomial and $B(n;x)$ is a
linear polynomial with leading coefficient $B_{1}(n;c)$,%
\begin{eqnarray*}
||\widehat{S}_{n}^{M,N}||_{S}^{2} &=&\langle (x-c)^{2}\widehat{S}%
_{n}^{M,N}(x),\widehat{\Pi }_{n-2}(x)\rangle _{\alpha } \\
&=&\langle A(n;x)\widehat{L}_{n}^{\alpha }(x),\widehat{\Pi }_{n-2}(x)\rangle
_{\alpha }+\langle B(n;x)\widehat{L}_{n-1}^{\alpha }(x),\widehat{\Pi }%
_{n-2}(x)\rangle _{\alpha } \\
&=&\langle \widehat{L}_{n}^{\alpha }(x),x^{n}\rangle _{\alpha
}+B_{1}(n;c)\langle \widehat{L}_{n-1}^{\alpha }(x),x^{n-1}\rangle _{\alpha }.
\end{eqnarray*}%
The first term in the above expression is the norm of the monic Laguerre
polynomial of degree $n$ and the second one is the norm of the Laguerre
polynomial of degree $n-1$ times $B_{1}(n;c)$, which is given in (\ref%
{[Sec3]-CoefA0A1B0B1}). This means%
\begin{equation*}
||\widehat{S}_{n}^{M,N}||_{S}^{2}=||\widehat{L}_{n}^{\alpha }||_{\alpha
}^{2}+B_{1}(n;c)||\widehat{L}_{n-1}^{\alpha }||_{\alpha }^{2}.
\end{equation*}%

Using the estimates \eqref{[Sec2]-LagNorm} and Proposition \ref{[Sec3]behavA0A1B0B1}, we obtain
\begin{equation*}
||\widehat{S}_{n}^{M,N}||_{S}^{2}\sim \Gamma (n+1)\Gamma (n+\alpha +1),
\end{equation*}%
which completes the proof.
\end{proof}

We are ready to find the five-term recurrence relation satisfied by $%
\widehat{S}_{n}^{M,N}(x)$, and the asymptotic behavior of the corresponding
coefficients. Next, we will focus our attention on its proof.

Let consider the Fourier expansion of $(x-c)^{2}\widehat{S}_{n}^{M,N}(x)$ in
terms of $\{\widehat{S}_{n}^{M,N}(x)\}_{n\geq 0}$%
\begin{equation*}
(x-c)^{2}\widehat{S}_{n}^{M,N}(x)=\sum_{k=0}^{n+2}\lambda _{n,k}\widehat{S}%
_{k}^{M,N}(x),
\end{equation*}%
where%
\begin{equation}
\lambda _{n,k}=\frac{\langle (x-c)^{2}\widehat{S}_{n}^{M,N}(x),\widehat{S}%
_{k}^{M,N}(x)\rangle _{S}}{||\widehat{S}_{k}^{M,N}||_{S}^{2}},\quad
k=0,\ldots ,n+2.  \label{[Sec4]-CoefsS1}
\end{equation}%
Thus, $\lambda _{n,k}=0$ for $k=0,\ldots ,n-3$. We are dealing with monic
polynomials, so the leading coefficient $\lambda _{n,n+2} =1$.

To obtain $\lambda _{n,n+1}$, we use the connection formula (\ref%
{[Sec3]-ConnFormS}), with coefficients $A(n;x)$ and $B(n;x)$ as in (\ref%
{[Sec3]-ABxn}). Thus,%
\begin{eqnarray*}
\lambda _{n,n+1}
&=&\frac{1}{||\widehat{S}_{n+1}^{M,N}||_{S}^{2}}\langle A(n;x)\widehat{L}%
_{n}^{\alpha }(x),\widehat{S}_{n+1}^{M,N}(x)\rangle _{S}+\frac{1}{||\widehat{%
S}_{n+1}^{M,N}||_{S}^{2}}\langle B(n;x)\widehat{L}_{n-1}^{\alpha }(x),%
\widehat{S}_{n+1}^{M,N}(x)\rangle _{S} \\
&=&\frac{1}{||\widehat{S}_{n+1}^{M,N}||_{S}^{2}}\langle (x-c)^{2}\widehat{L}%
_{n}^{\alpha }(x),\widehat{S}_{n+1}^{M,N}(x)\rangle _{S}+A_{1}(n;c).
\end{eqnarray*}%
Let us study the discrete Sobolev inner product $\langle (x-c)^{2}\widehat{L}%
_{n}^{\alpha }(x),\widehat{S}_{n+1}^{M,N}(x)\rangle _{S}$ above. Applying (%
\ref{[Sec4]-SymmInn}), (\ref{[Sec4]-Property2}), (\ref{[Sec2]-LagNorm}) and
Lemma \ref{[Sec4]-LEMMA-41}, we obtain%
\begin{eqnarray*}
\langle (x-c)^{2}\widehat{L}_{n}^{\alpha }(x),\widehat{S}_{n+1}^{M,N}(x)%
\rangle _{S} &=&\langle \widehat{L}_{n}^{\alpha }(x),(x-c)^{2}\widehat{S}%
_{n+1}^{M,N}(x)\rangle _{\alpha } \\
&=&\tilde{d}_{n+1}\,||\widehat{L}_{n}^{\alpha }||_{\alpha }^{2}.
\end{eqnarray*}%
From (\ref{[Sec2]-CocGammas}), Lemma \ref{[Sec4]-LEMMA-42} and Proposition \ref{[Sec3]propA0A1B0B1}

\begin{equation*}
\lambda _{n,n+1}=\frac{\tilde{d}_{n+1}\,||\widehat{L}_{n}^{\alpha
}||_{\alpha }^{2}+A_{1}(n;c)}{||\widehat{S}_{n+1}^{M,N}||_{S}^{2}}\sim 4n.
\end{equation*}

In order to compute $\lambda _{n,n}$, from (\ref{[Sec3]-ConnFormS}) and (\ref%
{[Sec3]-ABxn}) we get%
\begin{eqnarray*}
\lambda _{n,n} &=&\frac{\langle (x-c)^{2}\widehat{L}_{n}^{\alpha }(x),%
\widehat{S}_{n}^{M,N}(x)\rangle _{S}}{||\widehat{S}_{n}^{M,N}||_{S}^{2}}%
+A_{1}(n;c)\frac{\langle (x-c)\widehat{L}_{n}^{\alpha }(x),\widehat{S}%
_{n}^{M,N}(x)\rangle _{S}}{||\widehat{S}_{n}^{M,N}||_{S}^{2}} \\
&&+A_{0}(n;c)+B_{1}(n;c).
\end{eqnarray*}%
But, according to (\ref{[Sec4]-SymmInn}), (\ref{[Sec4]-Property2}) and Lemma %
\ref{[Sec4]-LEMMA-41}, the first term is%
\begin{equation*}
\frac{\langle (x-c)^{2}\widehat{L}_{n}^{\alpha }(x),\widehat{S}%
_{n}^{M,N}(x)\rangle _{S}}{||\widehat{S}_{n}^{M,N}||_{S}^{2}}=\tilde{c}_{n}%
\frac{||\widehat{L}_{n}^{\alpha }||_{\alpha }^{2}}{||\widehat{S}%
_{n}^{M,N}||_{S}^{2}}.
\end{equation*}%
After some algebraic manipulations, from (\ref{[Sec2]-3TRRLag}) we get%
\begin{equation*}
(x-c)\widehat{L}_{n}^{\alpha }(x)=(x-c)^{2}\widehat{L}_{n-1}^{\alpha
}(x)-(\beta _{n-1}-c)(x-c)\widehat{L}_{n-1}^{\alpha }(x)-\gamma _{n-1}(x-c)%
\widehat{L}_{n-2}^{\alpha }(x),
\end{equation*}%
Using this expression, we obtain%
\begin{eqnarray*}
\frac{\langle (x-c)\widehat{L}_{n}^{\alpha }(x),\widehat{S}%
_{n}^{M,N}(x)\rangle _{S}}{||\widehat{S}_{n}^{M,N}||_{S}^{2}} &=&\frac{%
\langle \widehat{L}_{n-1}^{\alpha }(x),(x-c)^{2}\widehat{S}%
_{n}^{M,N}(x)\rangle _{\alpha }}{||\widehat{S}_{n}^{M,N}||_{S}^{2}}-(\beta
_{n-1}-c) \\
&=&\tilde{d}_{n}\frac{||\widehat{L}_{n-1}^{\alpha }||_{\alpha }^{2}}{||%
\widehat{S}_{n}^{M,N}||_{S}^{2}}-(\beta _{n-1}-c).
\end{eqnarray*}%
As a consequence, we get%
\begin{eqnarray*}
\lambda _{n,n} &=&\frac{\tilde{c}_{n}\,||\widehat{L}_{n}^{\alpha }||_{\alpha
}^{2}+\tilde{d}_{n}\,||\widehat{L}_{n-1}^{\alpha }||_{\alpha }^{2}+(\beta
_{n-1}-c)+A_{0}(n;c)+B_{1}(n;c)}{||\widehat{S}_{n}^{M,N}||_{S}^{2}} \\
&\sim &6n^{2}.
\end{eqnarray*}

A similar analysis yields%
\begin{eqnarray*}
\lambda _{n,n-1} &=&\frac{\tilde{d}_{n}\,||\widehat{L}_{n-1}^{\alpha
}||_{\alpha }^{2}+A_{1}(n-1;c)||\widehat{S}_{n}^{M,N}||_{S}^{2}}{||\widehat{S%
}_{n-1}^{M,N}||_{S}^{2}}\sim 4n^{3}, \\
\lambda _{n,n-2} &=&\frac{||\widehat{S}_{n}^{M,N}||_{S}^{2}}{||\widehat{S}%
_{n-2}^{M,N}||_{S}^{2}}\sim n^{4}.
\end{eqnarray*}

We can summarize the results of this Section in the following theorem.

\begin{theorem}[Five-term recurrence relation]
For every $n\geq 1$, $\alpha >-1$ and $c\in \mathbb{R}_{+}$, the monic
Laguerre-Sobolev type polynomials $\{\widehat{S}_{n}^{M,N}\}_{n\geq 0}$,
orthogonal with respect to (\ref{[Sec1]-DicrLagSob}) satisfy the following
five-term recurrence relation%
\begin{equation*}
(x-c)^{2}\widehat{S}_{n}^{M,N}(x)=
\end{equation*}%
\begin{equation*}
\widehat{S}_{n+2}^{M,N}(x)+\lambda _{n,n+1}\widehat{S}_{n+1}^{M,N}(x)+%
\lambda _{n,n}\widehat{S}_{n}^{M,N}(x)+\lambda _{n,n-1}\widehat{S}%
_{n-1}^{M,N}(x)+\lambda _{n,n-2}\widehat{S}_{n-2}^{M,N}(x),
\end{equation*}%
with%
\begin{equation*}
\lambda _{n,n+1}=\frac{\tilde{d}_{n+1}\,||\widehat{L}_{n}^{\alpha
}||_{\alpha }^{2}+A_{1}(n;c)}{||\widehat{S}_{n+1}^{M,N}||_{S}^{2}}\sim
4n,
\end{equation*}%
\begin{equation*}
\lambda _{n,n}=\frac{\tilde{c}_{n}\,||\widehat{L}_{n}^{\alpha }||_{\alpha
}^{2}+\tilde{d}_{n}\,||\widehat{L}_{n-1}^{\alpha }||_{\alpha }^{2}-(\beta
_{n-1}-c)+A_{0}(n;c)+B_{1}(n;c)}{||\widehat{S}_{n}^{M,N}||_{S}^{2}}\sim
6n^{2},
\end{equation*}%
\begin{equation*}
\lambda _{n,n-1}=\frac{\tilde{d}_{n}\,||\widehat{L}_{n-1}^{\alpha
}||_{\alpha }^{2}+A_{1}(n-1;c)||\widehat{S}_{n}^{M,N}||_{S}^{2}}{||\widehat{S%
}_{n-1}^{M,N}||_{S}^{2}}\sim 4n^{3},
\end{equation*}%
\begin{equation*}
\lambda _{n,n-2}=\frac{||\widehat{S}_{n}^{M,N}||_{S}^{2}}{||\widehat{S}%
_{n-2}^{M,N}||_{S}^{2}}\sim n^{4}.
\end{equation*}
\end{theorem}

\section*{Acknowledgements}
The authors thank the reviewers for their careful revision of the manuscript. Their
helpful comments and suggestions contributed to improve substantially style and presentation of the manuscript.

%%%%%%%%%%%%%%%%%%%%%%%%%%%%%%%%%%%%%%%%%%%%%%%%%%%%%%%%%%%%%%%%%%%%%%%%%%%%%%%%%%

\end{document}